\documentclass[11pt]{amsart}

\usepackage[T1]{fontenc}
\usepackage[utf8]{inputenc}
\usepackage{lmodern}
\usepackage{amsmath,amssymb,amsthm,mathtools}
\usepackage{microtype}
\usepackage[colorlinks=true,linkcolor=blue,citecolor=blue,urlcolor=blue]{hyperref}

\newtheorem{theorem}{Theorem}[section]

\newtheorem{lemma}[theorem]{Lemma}
\newtheorem{corollary}[theorem]{Corollary}
\newtheorem{fact}[theorem]{Fact}
\theoremstyle{definition}
\newtheorem{definition}[theorem]{Definition}

\theoremstyle{remark}
\newtheorem{remark}[theorem]{Remark}

\newcommand{\om}{\omega}
\newcommand{\omone}{\omega_1}
\newcommand{\Om}{\Omega}
\newcommand{\PP}{\mathbb P}
\newcommand{\QQ}{\mathbb Q}

\newcommand{\II}{\mathcal I}
\newcommand{\RO}{\operatorname{RO}}
\newcommand{\otp}{\operatorname{otp}}
\newcommand{\MA}{\operatorname{MA}}
\newcommand{\forces}{\Vdash}
\newcommand{\notarrow}{\nrightarrow}
\newcommand{\frakb}{\mathfrak b}
\newcommand{\frakd}{\mathfrak d}
\newcommand{\frakc}{\mathfrak c}

\title[Martin's axiom and $\omega_1^2 \longrightarrow (\omega_1^2, 3)^2$]
{Martin's axiom and $\omega_1^2 \longrightarrow (\omega_1^2, 3)^2$}
\author {Mohammad Golshani}
\address{School of Mathematics\\
 Institute for Research in Fundamental Sciences (IPM)\\
  P.O. Box:
19395-5746\\
 Tehran-Iran.}
\email{golshani.m@gmail.com}
\urladdr{http://math.ipm.ac.ir/~golshani/}

\begin{document}

\begin{abstract}
Starting with CH and Hajnal's coloring, we show that a standard finite-support
iteration of \(\sigma\)-centered forcing notions gives a model of
\[
 \MA_{\omone}(\sigma\text{-centered})+2^{\aleph_0}=\aleph_2
 +\omega_1^2\notarrow(\omega_1^2,3)^2.
\]
We also isolate a simple
forcing-preservation principle: the same ground-model coloring remains
a witness after forcing with any poset whose subfamilies of size at most
\(\omone\) are countable unions of linked sets.  Under
\(\MA_{\omone}\), every c.c.c. forcing has this local property, so every
existing witness is preserved by every c.c.c. forcing over that model. 
\end{abstract}

\maketitle

\section{Introduction}
The arrow notation used here goes back to Rado and the development of
the partition calculus by Erd\H{o}s, Hajnal, and Rado. 
For ordinals \(\alpha,\beta,\gamma\), the relation
\[
  \alpha\longrightarrow(\beta,\gamma)^2
\]
means that every coloring \(c:[\alpha]^2\to 2\) has either a
color-\(0\) homogeneous subset of order type \(\beta\), or a
color-\(1\) homogeneous subset of order type \(\gamma\).  Equivalently,
every graph on \(\alpha\) has an independent set of order type \(\beta\)
or a clique of order type \(\gamma\).

We are concerned with the Erd\H{o}s--Hajnal problem
\begin{equation}\label{eq:main-positive}
  \omone^2\longrightarrow(\omone^2,3)^2.
\end{equation}
Here and below \(\omone^2\) denotes ordinal multiplication
\(\omone\cdot\omone\), not merely the cardinal \(\aleph_1\). Hajnal proved that CH implies the failure of
\eqref{eq:main-positive} \cite{Hajnal1971}.  
The construction organizes a triangle-free graph using a set mapping and
an enumeration of the relevant small configurations.  CH supplies the
enumeration at length \(\omone\).
Larson revisited Hajnal's construction in 1998 and replaced CH by the
existence of a short scale \cite{Larson1998}, in particular, she showed that.
\begin{equation}\label{eq:larson}
  \frakd=\aleph_1\quad\Longrightarrow\quad
  \omone^2\notarrow(\omone^2,3)^2.
\end{equation}
There are positive forcing-axiom results for related, smaller ordinal
partition problems.  For example, Baumgartner proved, from a suitable
form of Martin's Axiom, that
\(\omone\cdot\om\to(\omone\cdot\om,3)^2\); see
\cite{Baumgartner1989,ErdosProblem1171}.

Erd\H{o}s and Hajnal have
asked whether
\[
  \MA_{\aleph_1}+2^{\aleph_0}=\aleph_2
\]
implies \eqref{eq:main-positive}; see
\cite{ErdosHajnal1971,ErdosHajnal1974,ErdosProblemPage,Komjath2025}.
The question is still  open.  Baumgartner asked  whether PFA implies \eqref{eq:main-positive}.

In this paper we study this question and obtain some results around it.
The first main result of the paper is the following.
\begin{theorem}\label{thm:intro-ma}
Assume \(\MA_{\aleph_1}\).  If \(c\) is a witness to
\(\omone^2\notarrow(\omone^2,3)^2\), then every c.c.c. forcing preserves
that same witness.
\end{theorem}
We then prove the following partial answer to the question of Erd\H{o}s and Hajnal.

\begin{theorem}[Relative consistency]\label{thm:intro-consistency}
If ZFC is consistent, then so is
\[
 \mathrm{ZFC}+\MA_{\aleph_1}(\sigma\text{-centered})
 +2^{\aleph_0}=\aleph_2
 +\omone^2\notarrow(\omone^2,3)^2.
\]
The model can moreover satisfy
\(\frakb=\frakd=\frakc=\aleph_2\).
\end{theorem}

\section{A preservation theorem}
Set
\[
  \Om=\omone\cdot\omone
\]
and divide \(\Om\) into its canonical rows
\[
  R_\xi=[\omone\cdot\xi,\omone\cdot(\xi+1))
  \qquad(\xi<\omone).
\]
Define
\[
  \II=\{X\subseteq\Om:\otp(X)<\Om\}.
\]
The proof of the preservation theorem rests on the following elementary
description of \(\II\).

\begin{lemma}\label{lem:rows}
For \(X\subseteq\Om\), the following are equivalent.
\begin{enumerate}
\item \(\otp(X)=\Om\).
\item The set
\[
 S_X=\{\xi<\omone:|X\cap R_\xi|=\aleph_1\}
\]
is uncountable.
\end{enumerate}
Consequently, \(\II\) is a \(\sigma\)-ideal.
\end{lemma}
\iffalse
\begin{proof}
If \(S_X\) is uncountable, then it has order type \(\omone\).  For each
\(\xi\in S_X\), the set \(X\cap R_\xi\), in its inherited order, has
order type \(\omone\).  Hence \(X\) contains the ordinal sum of
\(\omone\) many copies of \(\omone\), and therefore has order type
\(\Om\).

Conversely, suppose that \(S_X\) is countable.  Choose
\(\delta<\omone\) above \(S_X\).  The part of \(X\) below
\(\omone\cdot\delta\) has order type below \(\Om\).  Above that point,
every row intersection is countable.  The ordinal sum of \(\omone\) many
countable ordinals has order type at most \(\omone\): every proper
initial sum is countable, and the full sum is their \(\omone\)-sequence
of suprema.  It follows that \(\otp(X)<\Om\).

For the final assertion, let \(X=\bigcup_{n<\om}X_n\) with each
\(X_n\in\II\).  By the characterization, each \(S_{X_n}\) is countable.
If \(\xi\notin\bigcup_nS_{X_n}\), then every \(X_n\cap R_\xi\) is
countable, so \(X\cap R_\xi\) is countable.  Thus
\(S_X\subseteq\bigcup_nS_{X_n}\) is countable and \(X\in\II\).
\end{proof}

\begin{remark}
The \(\sigma\)-ideal property is special to the two-level ordinal
\(\omone\cdot\omone\) and is the engine of the forcing argument.  Merely
knowing that the family of short sets is closed under finite unions would
not suffice.
\end{remark}
\fi

We use the convention that \(q\leq p\) means that \(q\) is the stronger
condition.

\begin{definition}
A subset \(L\) of a forcing notion is \emph{linked} if every two members
of \(L\) are compatible.  A forcing notion \(\PP\) is
\emph{locally \(\omone\)-\(\sigma\)-linked} if every
\(A\subseteq\PP\) of cardinality at most \(\aleph_1\) is the union of
countably many linked subsets.
\end{definition}

Clearly every \(\sigma\)-linked forcing is locally
\(\omone\)-\(\sigma\)-linked. 

\begin{theorem}\label{thm:preservation}
Let \(c:[\Om]^2\to2\) witness
\(\Om\notarrow(\Om,3)^2\).  If \(\PP\) is locally
\(\omone\)-\(\sigma\)-linked, then the same coloring \(c\) witnesses the
negative relation in every \(\PP\)-generic extension.
\end{theorem}

\begin{proof}
Suppose toward a contradiction that some \(p\in\PP\) and some name
\(\dot H\) satisfy
\[
 p\forces ``\dot H\subseteq\check\Om,
 \ \otp(\dot H)=\check\Om,
 \text{ and }\dot H\text{ is } 0\text{- homogeneous for }c.''
\]
For every \(x<\Om\) for which this is possible, choose
\(p_x\leq p\) with
\[
 p_x\forces \check x\in\dot H,
\]
and put
\[
 X=\{x<\Om:p_x\text{ was chosen}\}.
\]
Then \(p\forces\dot H\subseteq\check X\).  Hence \(X\notin\II\), for if
\(\otp(X)<\Om\), every subset of \(X\), in any forcing extension, has
order type below \(\Om\).

The family \(\{p_x:x\in X\}\) has cardinality at most \(\aleph_1\).
By local \(\omone\)-\(\sigma\)-linkedness, write it as
\(\bigcup_{n<\om}L_n\), where every \(L_n\) is linked, and set
\[
 X_n=\{x\in X:p_x\in L_n\}.
\]
Since \(X=\bigcup_nX_n\) and \(\II\) is a \(\sigma\)-ideal, some
\(X_n\notin\II\).  If \(x\ne y\) belong to this \(X_n\), then
\(p_x\) and \(p_y\) are compatible.  A common extension forces both
\(x\) and \(y\) into \(\dot H\), so necessarily \(c(\{x,y\})=0\).
Therefore \(X_n\) is a ground-model \(0\)-homogeneous set of order
type \(\Om\), contrary to the choice of \(c\).
\end{proof}

\begin{corollary}\label{cor:sigmalinked}
Every \(\sigma\)-linked, and hence every \(\sigma\)-centered, forcing
preserves every ground-model witness to
\(\omone^2\notarrow(\omone^2,3)^2\).
\end{corollary}
Now we recall a standard consequence of Martin's Axiom; see, for example,
\cite[Chapter 2]{BartoszynskiJudah1995}.

\begin{fact}\label{fact:ma-small}
Assume \(\MA_\kappa\).  Every c.c.c. forcing notion of cardinality at
most \(\kappa\) is \(\sigma\)-centered.
\end{fact}
The following is an immediate corollary of the above fact.
\begin{lemma}\label{lem:ma-local}
Assume \(\MA_{\aleph_1}\).  Every c.c.c. forcing notion is locally
\(\omone\)-\(\sigma\)-centered, and hence locally
\(\omone\)-\(\sigma\)-linked.
\end{lemma}

\begin{proof}
Let \(A\subseteq\PP\) have size at most \(\aleph_1\), where \(\PP\) is
c.c.c.   Let \(B\) be the
Boolean subalgebra of  \(\RO(\PP)\) generated by the regular-open values of the members of
\(A\).  Then \(|B|\leq\aleph_1\), and \(B^+\) is c.c.c.  By Fact
\ref{fact:ma-small}, write
\[
 B^+=\bigcup_{n<\om}C_n
\]
with each \(C_n\) centered.  Pulling the \(C_n\)'s back to \(A\) gives a
countable cover of \(A\) by subfamilies every finite part of which has a
common extension in \(\PP\).
\end{proof}

Theorem \ref{thm:intro-ma} is now immediate from Lemma
\ref{lem:ma-local} and Theorem \ref{thm:preservation}.

\begin{corollary}\label{cor:ma-pres}
In a model of \(\MA_{\aleph_1}\), every c.c.c. forcing preserves every
already existing witness to
\(\omone^2\notarrow(\omone^2,3)^2\).
\end{corollary}

\section{A restricted forcing-axiom model}

We next prove Theorem \ref{thm:intro-consistency}.  We use the standard
iteration fact that a finite-support iteration of
\(\sigma\)-centered forcings of length strictly below
\(\frakc^+\), where \(\frakc\) is computed in the ground model, is
\(\sigma\)-centered; see
\cite{Tall1994,FischerKoelbingWohofsky2023}.

\begin{theorem}\label{thm:restricted-model}
There is, relative to ZFC, a model satisfying
\[
 \MA_{\aleph_1}(\sigma\text{-centered})
 +\frakc=\aleph_2
 +\frakb=\frakd=\frakc=\aleph_2.+\omone^2\notarrow(\omone^2,3)^2.
\]
\end{theorem}

\begin{proof}
Start with a model of GCH.  By Hajnal's theorem, fix a coloring
\(c:[\Om]^2\to2\) witnessing the negative relation.
Construct a finite-support iteration
\[
 \langle\PP_\alpha,\dot\QQ_\alpha:\alpha<\om_2\rangle
\]
of \(\sigma\)-centered forcing notions.  Use the usual bookkeeping so
that every name for a \(\sigma\)-centered forcing together with a family
of at most \(\aleph_1\) dense subsets is handled at a later stage.  Put a
Cohen forcing at cofinally many stages.  The standard nice-name argument
then gives
\[
 \forces_{\PP_{\om_2}}
 ``\MA_{\aleph_1}(\sigma\text{-centered})
 \text{ and }\frakc=\aleph_2.''
\]
It remains to verify that \(c\) survives.  For every
\(\alpha<\om_2\), the ordinal \(\alpha\) has cardinality at most
\(\aleph_1=\frakc^V\).  The iteration theorem quoted above therefore
implies that \(\PP_\alpha\) is \(\sigma\)-centered.

Now let \(A\subseteq\PP_{\om_2}\) have cardinality at most
\(\aleph_1\).  By standard arguments,  \(A\subseteq\PP_\alpha\) for some
\(\alpha<\om_2\).  Since \(\PP_\alpha\) is \(\sigma\)-centered, so is
\(A\).  Hence the final forcing \(\PP_{\om_2}\) is locally
\(\omone\)-\(\sigma\)-centered.  Theorem \ref{thm:preservation} shows
that \(c\) remains a negative witness in the final extension.

Finally note that $\MA_{\aleph_1}(\sigma\text{-centered})
 +\frakc=\aleph_2$   implies $\frakb=\frakd=\frakc=\aleph_2,$ which completes the proof.
\end{proof}
Combining Theorem \ref{thm:restricted-model} with Larson's theorem
\eqref{eq:larson} gives a useful methodological conclusion.

\begin{corollary}\label{cor:not-short-scale}
It is consistent that
\[
 \frakd=\aleph_2
 \quad\text{and}\quad
 \omone^2\notarrow(\omone^2,3)^2.
\]
Consequently, the implication
\[
 \omone^2\notarrow(\omone^2,3)^2
 \quad\Longrightarrow\quad
 \frakd=\aleph_1
\]
is not provable in ZFC.
\end{corollary}


\begin{thebibliography}{99}

\bibitem{BartoszynskiJudah1995}
T.~Bartoszy\'nski and H.~Judah,
\emph{Set Theory: On the Structure of the Real Line},
A K Peters, Wellesley, MA, 1995.

\bibitem{Baumgartner1989}
J.~E.~Baumgartner,
Remarks on partition ordinals,
in \emph{Set Theory and its Applications},
Lecture Notes in Mathematics 1401, Springer, 1989, pp.~5--17.
\href{https://doi.org/10.1007/BFb0097328}
{doi:10.1007/BFb0097328}.

\bibitem{ErdosProblem1171}
Erd\H{o}s Problems,
Problem 1171,
\url{https://www.erdosproblems.com/1171}, accessed 13 August 2026.

\bibitem{ErdosHajnal1971}
P.~Erd\H{o}s and A.~Hajnal,
Unsolved problems in set theory,
in \emph{Axiomatic Set Theory},
Proc. Sympos. Pure Math. XIII, Part I,
Amer. Math. Soc., 1971, pp.~17--48.

\bibitem{ErdosHajnal1974}
P.~Erd\H{o}s and A.~Hajnal,
Unsolved and solved problems in set theory,
in \emph{Proceedings of the Tarski Symposium},
Proc. Sympos. Pure Math. XXV,
Amer. Math. Soc., 1974, pp.~269--287.

\bibitem{ErdosProblemPage}
Erd\H{o}s Problems,
Ordinal Ramsey: \(\omone^2\to(\omone^2,3)^2\),
\url{https://mathweb.ucsd.edu/~erdosproblems/erdos/newproblems/OrdinalRamsey4.html},
accessed 13 August 2026.

\bibitem{FischerKoelbingWohofsky2023}
V.~Fischer, M.~Koelbing, and W.~Wohofsky,
Towers, mad families, and unboundedness,
\emph{Arch. Math. Logic} \textbf{62} (2023), 811--830.
\href{https://doi.org/10.1007/s00153-023-00861-x}
{doi:10.1007/s00153-023-00861-x}.

\bibitem{Hajnal1971}
A.~Hajnal,
A negative partition relation,
\emph{Proc. Nat. Acad. Sci. U.S.A.} \textbf{68} (1971), 142--144.
\href{https://doi.org/10.1073/pnas.68.1.142}
{doi:10.1073/pnas.68.1.142}.

\bibitem{Komjath2025}
P.~Komj\'ath,
The Erd\H{o}s--Hajnal problem list,
\emph{Bull. Symbolic Logic} \textbf{31} (2025), 418--461.
\href{https://doi.org/10.1017/bsl.2025.1}
{doi:10.1017/bsl.2025.1}.

\bibitem{Larson1998}
J.~A.~Larson,
An ordinal partition from a scale,
in C.~A.~Di Prisco, J.~A.~Larson, J.~Bagaria, and A.~R.~D.~Mathias
(eds.), \emph{Set Theory},
Kluwer, 1998, pp.~109--125.
\href{https://doi.org/10.1007/978-94-015-8988-8_8}
{doi:10.1007/978-94-015-8988-8\_8}.

\bibitem{Tall1994}
F.~D.~Tall,
\(\sigma\)-centred forcing and reflection of (sub)metrizability,
\emph{Proc. Amer. Math. Soc.} \textbf{121} (1994), 299--306.

\end{thebibliography}
\end{document}